\documentclass{amsart}
\usepackage{graphicx}
\usepackage{amsmath}
\usepackage{mathrsfs}
\usepackage{amsthm}
\usepackage{amsthm}
\usepackage{amssymb}
\numberwithin{equation}{section}
\newtheorem{thr}{Theorem}[section]
\theoremstyle{definition}
\newtheorem{defn}[thr]{Definition}
\theoremstyle{remark}
\newtheorem{rem}{Remark}
\theoremstyle{plain}
\newtheorem{lem}[thr]{Lemma}
\theoremstyle{plain}
\newtheorem{prop}[thr]{Proposition}
\theoremstyle{plain}
\newtheorem{cor}[thr]{Corollary}
\newcommand{\abs}[1]{\lvert#1\rvert}
\newcommand{\tab}{\mathrm{t}}

\DeclareMathOperator{\STD}{\mathrm{STD}}
\DeclareMathOperator{\Tab}{\mathrm{Tab}}
\DeclareMathOperator{\shape}{\mathrm{Shape}}
\DeclareMathOperator{\SD}{\mathrm{SD}}
\DeclareMathOperator{\SA}{\mathrm{SA}}
\DeclareMathOperator{\WD}{\mathrm{WD}}
\DeclareMathOperator{\WA}{\mathrm{WA}}
\DeclareMathOperator{\D}{\mathrm{D}}
\DeclareMathOperator{\A}{\mathrm{A}}
\DeclareMathOperator{\Pos}{\mathrm{Pos}}

\newcommand{\row}{\mathrm{row}}
\newcommand{\col}{\mathrm{col}}

\usepackage{young}

\theoremstyle{remark}

\begin{document}    

\title[\unboldmath $W\!$-\,graph ideals II.]{\boldmath $W\!$-\,graph ideals II}\label{Pre}

\author{\emph{\textbf{Van Minh Nguyen}}}

\begin{abstract}
In \cite{howvan:wgraphDetSets}, the concept of a \(W\!\)-\,graph
ideal in a Coxeter group was introduced, and it was shown how a
\(W\!\)-\,graph can be constructed from a given \(W\!\)-\,graph
ideal. In this paper, we describe a class of \(W\!\)-\,graph
ideals from which certain Kazhdan-Lusztig left cells arise. The
result justifies the algorithm as illustrated in
\cite{howvan:wgraphDetSets} for the construction of
\(W\!\)-\,graphs for Specht modules for the Hecke algebra of type
\(A\).
\end{abstract}

\maketitle

\section{Introduction}

Let \((W,S)\) be a Coxeter system and \(\mathcal{H}(W)\) its Hecke
algebra over \(\mathbb{Z}[q,q^{-1}]\), the ring of Laurent
polynomials in the indeterminate \(q\). In
\cite{howvan:wgraphDetSets}, we introduced the concept of a
\textit{\(W\!\)-graph ideal\/} in \((W,\le_L)\) with respect to a
subset \(J\) of \(S\), where \(\le_L\) is the left weak Bruhat
order on \(W\!\), and gave a Kazhdan-Lusztig like algorithm to
produce, for any such ideal~\(\mathscr{I}\!\), a \(W\!\)-graph
with vertices indexed by the elements of~\(\mathscr{I}\!\). In
particular, \(W\) itself is a \(W\!\)-graph ideal with respect to
\(\emptyset\), and the \(W\!\)-graph obtained is the
Kazhdan-Lusztig \(W\!\)-graph for the regular representation of
\(\mathcal{H}(W)\) (as defined in~\cite{kazlus:coxhecke}). More
generally, it was shown that if \(J\) is an arbitrary subset of
\(S\) then \(D_{J}\), the set of distinguished left coset
representatives of \(W_{J}\) in \(W\), is a \(W\!\)-graph ideal
with respect to~\(J\) and also with respect to~\(\emptyset\), and
Deodhar's parabolic analogues of the Kazhdan-Lusztig construction
are recovered. In this paper we continue the work in
\cite{howvan:wgraphDetSets}, and describe a larger class of
\(W\!\)-\,graph ideals for an arbitrary Coxeter group. Our main
aim is to show how to construct \(W\!\)-\,graphs for a wide class
of Kazhdan-Lusztig left cells, without having to first construct
the full Kazhdan-Lusztig \(W\!\)-graph corresponding to the
regular representation.

To this end we investigate conditions that are sufficient for a
sub-ideal of a given \(W\!\)-graph ideal (with respect to the left
weak order) to itself be a \(W\!\)-graph ideal. We find that if
the sub-ideal is a union of cells then it is a \(W\!\)-graph
ideal. It should be noted that, during the course of the proof, we
are required to verify a technical result that says that certain
structural constants of the associated \(\mathcal{H}(W)\)\,-module
are polynomials that are divisible by \(q\). In particular, for
the Kazhdan-Lusztig \(W\!\)-graph for the regular representation,
we find that if \(\mathcal{C}\) is the left cell that
contains~\(w_{J}\), the longest element of the finite standard
parabolic subgroup~\(W_J\), then \(\mathcal{C}w_{J}\) is a
\(W\!\)-graph ideal with respect to \(J\). Moreover, the
\(W\!\)-graph associated with the cell \(\mathcal{C}\) is
isomorphic to the \(W\!\)-graph constructed from the ideal
\(\mathcal{C}w_{J}\). The result shows that the algorithm in
\cite{howvan:wgraphDetSets} can be applied to construct
\(W\!\)-graphs for Kazhdan-Lusztig left cells that contain longest
elements of standard parabolic subgroups. In type \(A\), it is
known that each such cell is parametrized by the standard tableaux
of a fixed shape, and that the cell module is isomorphic to the
corresponding Specht module; hence the result justifies the
algorithm described in \cite{howvan:wgraphDetSets} for the
construction of \(W\!\)-graphs for Specht modules.

This paper is organized as follows. We start by recalling basic
definitions and facts concerning \(W\!\)-graphs; in particular,
cells and subquotients are discussed. In Section~3, we review the
notion of a \(W\!\)-graph ideal and the procedure for constructing
a \(W\!\)-graph from a \(W\!\)-graph ideal. Next, in Section~4, we
give a description of  \(W\!\)-graph ideals that arise from
sub-ideals of a given \(W\!\)-graph ideal, assuming certain
conditions. We deduce that if the cells of the associated
\(W\!\)-graph are ordered in the natural way, based on the
preorder by which cells are defined, then there is a unique
maximal cell, and this maximal cell is constructed from a
\(W\!\)-graph ideal. In Section~5 we show that the \(W\!\)-graph
for the left cell that contains the longest element of a standard
parabolic subgroup arises in the manner just described. The result
has some significant consequences when it is applied to Coxeter
groups of type~\(A\) and the associated Hecke algebras.
Specifically, it  justifies the way \(W\!\)-graphs for Specht
modules, which are exactly \(W\!\)-graphs for the corresponding
left cells, are calculated in \cite{howvan:wgraphDetSets}. These
topics are included in the discussion in Section 6.

\section{\(W\!\)-graphs, cells and subquotients}\label{Cox}

Since this is a continuation of the work in
\cite{howvan:wgraphDetSets}, we will assume the notation
introduced there. In particular, for a Coxeter system \((W,S)\),
we let \(l\) be its length function and let \(\leq\) and
\(\leq_{L}\) be its the Bruhat order and the left weak Bruhat
order respectively.

Let \(\mathcal{A} = \mathbb{Z}[q, q^{-1}]\), the ring of Laurent
polynomials with integer coefficients in the indeterminate \(q\),
and let \(\mathcal{A}^{+} = \mathbb{Z}[q]\), and let
\(\mathcal{H}(W) \) be the Hecke algebra associated with the
Coxeter system \((W, S)\). Our convention is that \(\mathcal{H}(W)
\) is the associative algebra over \(\mathcal{A}\) generated by
\(\{T_{s}\mid s\in S \}\), subject to the following defining
relations:
\begin{align*}
T^{2}_{s} &= 1 + (q -q^{-1})T_{s} \quad \text{for all \(s \in S\)},\\
T_{s}T_{s'}T_{s}\cdots &= T_{s'}T_{s}T_{s'}\cdots \quad \text{for
all \(s, s' \in S\)},
\end{align*}
where in the second of these there are \(m(s,s')\) factors on each
side, \(m(s,s')\) being the order of \(ss'\) in~\(W\). (We remark
that the traditional definition has \(T^{2}_{s}=q + (q -1)T_{s}\)
in place of the first relation above; our version is obtained by
replacing \(q\) by \(q^2\) and multiplying the generators
by~\(q^{-1}\).) It is well known that \(\mathcal{H}(W)\) is
\(\mathcal{A}\)-free with an \(\mathcal{A}\)-basis \((\, T_{w}\mid
w \in W\,)\) and multiplication satisfying
\begin{equation*}
T_{s}T_{w} = \begin{cases} T_{sw} & \text{if \(l(sw) > l(w)\),}\\
                            T_{sw} + (q - q^{-1})T_{w} & \text{if \(l(sw) < l(w)\).}
              \end{cases}
\end{equation*}
for all \(s\in S\) and \(w\in W\).

Let \(a \mapsto \overline{a}\) be the involutory automorphism of
\(\mathcal{A}=\mathbb{Z}[q,q^{-1}]\) defined by
\(\overline{q}=q^{-1}\). This extends to an involution on
\(\mathcal{H}(W)\) satisfying
\begin{equation*}
\overline{T_{s}} = T_{s}^{-1} = T_{s} - (q - q^{-1}) \quad
\text{for all \(s \in S\)}.
\end{equation*}

A \(W\!\)-graph \(\Gamma\) is a triple consisting of a set
\(V\!\), a function \(\mu\colon V \times V\to \mathbb{Z}\) and a
function \(\tau\) from \(V\) to the power set of \(S\), subject to
the requirement that the free \(\mathcal{A}\)-module with basis
\(V\) admits an \(\mathcal{H}\)-module structure satisfying
\begin{equation}\label{wgraphdef}
     T_{s}v = \begin{cases}
              -q^{-1}v \quad &\text{if \(s \in \tau(v)\)}\\
              qv + \sum_{\{u \in V \mid s \in \tau(u)\}}\mu(u,v)u
              \quad &\text{if \(s \notin \tau(v)\)},
     \end{cases}
\end{equation}
for all \(s \in S\) and \(v \in V\!\).

The set \(V\) is called the vertex set of \(\Gamma\), and there is
a directed edge from a vertex \(v\) to a vertex \(u\) if and only
if \(\mu(u,v) \neq 0\). When there is no ambiguity we may use the
notation \(\Gamma(V)\) for the \(W\!\)-graph with vertex
set~\(V\). We call the integer \(\mu(u,v)\) the \textit{weight\/}
of the edge from \(v\) to \(u\), and we call the set \(\tau(v)\)
the \textit{\(\tau\)-invariant\/} of the vertex~\(v\). The
\(\mathcal{H}(W)\)-module \(\mathcal{A}V\) encoded by \(\Gamma\)
will be denoted by \(M_{\Gamma}\).

Since \(M_{\Gamma}\) is \(\mathcal{A}\)-free it admits a unique
\(\mathcal{A}\)-semilinear involution \(\alpha \mapsto
\overline{\alpha}\) such that \(\overline v=v\) for all elements
\(v\) of the basis~\(V\). It follows from \eqref{wgraphdef} that
\(\overline{h\alpha} =\overline{h}\overline{\alpha}\) for all
\(h\in \mathcal{H}\) and \(\alpha \in M_{\Gamma}\).

Following \cite{kazlus:coxhecke}, define a preorder
\(\leq_{\Gamma}\) on the vertex set of \(\Gamma\) as follows: \(u
\leq_{\Gamma} v\) if there exists a sequence of vertices \(u =
x_{0},x_1,\ldots,,x_{m} = v\) such that \(\tau(x_{i-1}) \nsubseteq
\tau(x_{i})\) and \(\mu(x_{i-1},x_{i}) \neq 0\) for all \(i \in
[1,m]\). In other words, the preorder \(\leq_{\Gamma}\) is the
transitive closure of the relation \(\leftarrow_{\Gamma}\) on
\(V\) given by \(u\leftarrow_{\Gamma} v\) if \(\tau(u) \nsubseteq
\tau(v)\) and \(\mu(u,v) \neq 0\). Let \(\sim_{\Gamma}\) be the
equivalence relation corresponding to \(\leq_{\Gamma}\); that is,
\(u \sim_{\Gamma} v\) if and only if \( u \leq_{\Gamma}v\) and \(v
\leq_{\Gamma} u\). The equivalence classes with respect to
\(\sim_{\Gamma}\) are called the \textit{cells} of \(\Gamma\).
Each equivalence class, regarded as a full subgraph of \(\Gamma\),
is itself a \(W\!\)-\,graph, with the \(\mu\) and \(\tau\)
functions being the restrictions of those for \(\Gamma\). Thus, if
\(\mathcal{C}\) is a cell, then \(\Gamma(\mathcal{C}) =
(\mathcal{C}, \mu, \tau)\) is the \(W\!\)-\,graph associated with
the cell \(\mathcal{C}\). Observe that the preorder
\(\leq_{\Gamma}\) on the vertices induces a partial order on the
cells, via the rule that if \(\mathcal{C},\,\mathcal{C}'\) are
cells then \(\mathcal{C} \leq_{\Gamma} \mathcal{C}'\) if and only
if \(u \leq_{\Gamma} v\) for some (or, equivalently, all) \(u \in
\mathcal{C}\) and \(v \in \mathcal{C}'\).

Let \(U \subseteq V\). If \(U\) spans a
\(\mathcal{H}(W)\)-submodule of \(M_{\Gamma(V)}\), then \(U\) is
called a  \textit{closed subset} of \(V\). We see from Equation
(\ref{wgraphdef}) that this happens if and only if for all
vertices \(u\) and \(v\), if \(u\in U\) and \(v
\leftarrow_{\Gamma} u\) then \(v \in U\). (Note that in
\cite{stem:addwgraph} the term \textit{forward-closed} is used for
this concept.)

Provided that \(U\) is a closed subset of \(V\), the subgraphs
\(\Gamma(U)\) and \(\Gamma(V \setminus U)\) induced by \(U\) and
\(V \setminus U\) are themselves \(W\!\)-graphs, with edge weights
and \(\tau\) invariants inherited from \(\Gamma(V)\). Moreover, we
have
\begin{equation*}
M_{\Gamma(V \setminus U)} \cong M_{\Gamma(V)}/M_{\Gamma(U)} \quad
\text{as \(\mathcal{H}(W)\)\,-modules}.
\end{equation*}
If \(U_{2} \subseteq U_{1} \subseteq V\) is a nested sequence of
closed subsets of \(V\) then the \(W\!\)-graph \(\Gamma(U_{1}
\setminus U_{2})\) is called a \textit{subquotient} of
\(\Gamma(V)\), as the \(\mathcal{H}\)-module
\(M_{\Gamma(U_1/U_2)}\) is a quotient of a submodule of
\(M_{\Gamma(V)}\). It can be seen that if \(\Gamma(V)\) has no
non-empty proper subquotients then it consists of a single cell.

Let \(\Gamma(W) = (W, \mu, \tau)\) be the Kazhdan-Lusztig
\(W\!\)-\,graph, as defined in~\cite{kazlus:coxhecke}. Thus
\[
\mu(y,w)    = \begin{cases}
                    \mu_{y,w} &\quad \text{if \(y < w\)}\\
                    \mu_{w,y} &\quad \text{if \(w < y\)}
            \end{cases}\\
\]
where \(\mu_{y,w}\) is either zero or the leading coefficient of a
certain polynomial~\(P_{y,w}\), and
\[
\tau(w) = \mathcal{L}(w)= \{s \in S \mid l(sw) < l(w)\}.
\]
In fact Kazhdan and Lusztig show that \(W\) can be given the
structure of a \(W\times W^o\)-graph, where \(W^o\) is the
opposite of the group \(W\), but in the present paper we are
concerned only with the \(W\!\)-\,graph structure. The equivalence
classes determined by the preorder \(\le_{\Gamma(W)}\) (as defined
above) are called the \textit{left cells} of~\(W\).

\section{\(W\!\)-graph ideals}\label{section3}

Let \((W,S)\) be a Coxeter system and \(\mathcal{H}\) the
associated Hecke algebra. As in \cite{howvan:wgraphDetSets}, we
find it convenient to define \(\Pos(X)=\{\,s\in S\mid
l(xs)>l(x)\text{ for all }x\in X\,\}\), so that \(\Pos(X)\) is the
largest subset \(J\) of \(S\) such that \(X\subseteq D_J\). Let
\(\mathscr I\) be an ideal in the poset \((W,\leq_{L})\); that is,
\(\mathscr{I}\) is a subset of \(W\) such that every \(u\in W\)
that is a suffix of an element of \(\mathscr{I}\) is itself
in~\(\mathscr{I}\!\). This condition implies that
\(\Pos(\mathscr{I})=S\setminus\mathscr{I}= \{\,s\in S\mid
s\notin\mathscr{I}\,\}\). Let \(J\) be a subset of
\(\Pos(\mathscr{I})\),  so that \(\mathscr I\subseteq D_J\). For
each \(w \in \mathscr I\) we define the following subsets
of~\(S\):
\begin{align*}
\SA(\mathscr I,w)&=\{\,s\in S\mid sw>w\text{ and }sw\in\mathscr I\,\},\\
\SD(\mathscr I,w)&=\{\,s\in S\mid sw<w\,\},\\
\WA_{J}(\mathscr I,w)&=\{\,s\in S\mid sw> w\text{ and }
sw\in D_{J}\setminus\mathscr I\,\},\\
\WD_{J}(\mathscr I,w)&=\{\,s\in S\mid sw>w\text{ and }sw\notin
D_{J}\,\}.
\end{align*}
Since  \(\mathscr I\subseteq D_J\) it is clear that, for each \(w
\in \mathscr I\), each \(s \in S\) appears in exactly one of the
four sets defined above. We call the elements of these sets the
strong ascents, strong descents, weak ascents and weak descents
of~\(w\) relative to \(\mathscr{I}\) and~\(J\). In contexts where
the ideal \(\mathscr I\) and the set~\(J\) are fixed we may omit
reference to them and write, for example, \(\WA(w)\) rather than
\(\WA_{J}(\mathscr I,w)\). We also define the sets of descents and
ascents of \(w\) by \(\D_{J}(\mathscr I,w) = \SD(\mathscr I,w)
\cup \WD_{J}(\mathscr I,w)\) and \(\A_{J}(\mathscr I,w) =
\SA(\mathscr I,w) \cup \WA_{J}(\mathscr I,w)\).

\begin{rem}
It follows from Lemma \cite[Lemma 2.1. (iii)]{deo:paraKL} that
\begin{align*}
\WA(w) &= \{\,s \in S \mid sw \notin \mathscr I \text{ and } w^{-1}sw \notin J\,\}\\
\noalign{\vskip-6 pt\hbox{and}\vskip-6 pt} \WD(w) &= \{\,s \in S
\mid sw \notin \mathscr I \text{ and } w^{-1}sw \in J\,\},
\end{align*}
since \(sw \notin \mathscr I\) implies that \(sw>w\) (given that
\(\mathscr I\) is an ideal in \((W,\le_L)\)). Note also that
\(J=\WD(1)\).
\end{rem}

\begin{defn}\label{wgphdetelt}
With the above notation, the set \(\mathscr{I}\) is said to be a
\textit{\(W\!\)-graph ideal\/} with respect to \(J\) if the
following hypotheses are satisfied.
\begin{itemize}
\item[(i)] There exists an \(\mathcal{A}\)-free
\(\mathcal{H}\)-module \(\mathscr{S}=\mathscr{S}(\mathscr{I},J)\)
possessing an \(\mathcal{A}\)-basis \(B=(\,b_{w}\mid w\in\mathscr
I\,)\) on which the generators \(T_{s}\) act by
\begin{equation}\label{S_0action}
T_{s}b_{w} =
\begin{cases}
  b_{sw}  & \text{if \(s \in \SA(w)\),}\\
  b_{sw} + (q - q^{-1})b_{w} & \text{if \(s \in \SD(w)\),}\\
  -q^{-1}b_{w} & \text{if \(s \in \WD(w)\),}\\
  qb_{w} - \sum\limits_{\substack{y \in \mathscr I\\y < sw}} r^s_{y,w}b_{y} &
  \text{if \(s \in \WA(w)\),}
\end{cases}
\end{equation}
for some polynomials \(r^s_{y,w} \in q\mathcal{A}^{+}\!\).
\item[(ii)] The module \(\mathscr{S}\) admits an
\(\mathcal{A}\)-semilinear involution \(\alpha \mapsto
\overline{\alpha}\) satisfying \(\overline{b_1}=b_1\) and
\(\overline{h\alpha} =\overline{h}\overline{\alpha}\) for all
\(h\in \mathcal{H}\) and \(\alpha \in \mathscr{S}\).
\end{itemize}
\end{defn}

An obvious induction on \(l(w)\) shows that \(b_w=T_wb_1\) for all
\(w\in\mathscr I\).
\begin{defn}
If \(w\in W\) and \(\mathscr{I}=\{\,u\in W\mid u\le_Lw\,\}\) is a
\(W\!\)-graph ideal with respect to some \(J\subseteq S\) then we
call \(w\) a \textit{\(W\!\)-graph determining element}.
\end{defn}

\begin{rem}
It has been verified in \cite[Section 5]{howvan:wgraphDetSets}
that if \(W\) is finite then \(w_S\), the maximal length element
of~\(W\!\), is a \(W\!\)-graph determining element with respect to
\(\emptyset\), and \(d_{J}\), the minimal length element of the
left coset \(w_SW_{J}\), is a \(W\!\)-graph determining element
with respect to~\(J\) and also with respect to~\(\emptyset\).
\end{rem}

Let \(\mathscr I\) be a \(W\!\)-graph ideal with respect to
\(J\subseteq S\) and let \(\mathscr{S}(\mathscr{I},J)\) be the
corresponding \(\mathcal{H}\)-module (from
Definition~\ref{wgphdetelt}). From these data one can construct a
\(W\!\)-graph \(\Gamma\) with
\(M_\Gamma=\mathscr{S}(\mathscr{I},J)\). Specifically, the
following results are proved in~\cite{howvan:wgraphDetSets}.

\begin{lem}\cite[Lemma 7.2.]{howvan:wgraphDetSets}\label{uniCbasis1}
The \(\mathcal H\)-module \(\mathscr{S}(\mathscr{I},J)\) in
Definition~\ref{wgphdetelt} has a unique \(\mathcal{A}\)-basis \(C
= (\,c_{w} \mid w\in \mathscr I\,)\) such that for all \(w \in
\mathscr I\) we have \(\overline{c_{w}} = c_{w}\) and
\begin{equation}\label{qpoly1}
b_{w} = c_{w} + q\sum_{y < w} q_{y,w}c_{y}
\end{equation}
for certain polynomials \(q_{y,w} \in \mathcal{A}^{+}\).
\end{lem}
Define \(\mu_{y,w}\) to be the constant term of~\(q_{y,w}\). The
polynomials \(q_{y,w}\), where \(y < w\), can be computed
recursively by the following formulae.

\begin{cor}\cite[Corollary 7.4]{howvan:wgraphDetSets}\label{recursion}
Suppose that \(w<sw\in \mathscr I\) and \(y<sw\). If \(y=w\) then
\(q_{y,sw}=1\), and if \(y\ne w\) we have the following formulas:
\begin{itemize}
\item[(i)] \(q_{y,sw}=qq_{y,w}\) \ if \(s\in\A(y)\), \item[(ii)]
\(q_{y,sw} = -q^{-1}(q_{y,w}-\mu_{y,w})+q_{sy,w}
+\sum_x\mu_{y,x}q_{x,w}\) \ if \(s\in\SD(y)\), \item[(iii)]
\(q_{y,sw} = -q^{-1}(q_{y,w}-\mu_{y,w})+\sum_x\mu_{y,x}q_{x,w}\) \
if \(s\in\WD(y)\),
\end{itemize}
where \(q_{y,w}\) and \(\mu_{y,w}\) are regarded as \(0\) if
\(y\not<w\), and in \textup{(ii)} and \textup{(iii)} the sums
extend over all \(x\in \mathscr I\) such that \(y<x<w\) and
\(s\notin\D(x)\).
\end{cor}

Let \(\mu\colon C \times C \rightarrow \mathbb{Z}\) be given by
\begin{equation*}
\mu(c_{y},c_{w}) =
                     \begin{cases}
                         \mu_{y,w} &\text{if \(y < w\)}\\
                         \mu_{w,y} &\text{if \(w < y\)}\\
                         0 &\text{otherwise},
                     \end{cases}
\end{equation*}
and let \(\tau\) from \(C\) to the power set of~\(S\) be given by
\(\tau(c_{w}) = \D(w)\) for all~\(y\in\mathscr{I}\!\).
\begin{thr}\cite[Theorem 7.5.]{howvan:wgraphDetSets}\label{main-wg}
The triple \((C, \mu, \tau)\) is a \(W\!\)-graph.
\end{thr}

It is immediate from Corollary~\ref{recursion} that if
\(w<sw\in\mathscr I\) then \(\mu_{w,sw}=q_{w,sw}=1\), and since
also \(\D(sw)\nsubseteq \D(w)\) (since
\(s\in\D(sw)\setminus\D(w)\)) it follows that
\(c_{sw}\le_{\Gamma(C)}c_w\). A straightforward induction on
length now yields the first part of the following result, which in
turn immediately yields the second part.

\begin{cor}\label{edgeup}
\begin{itemize}
\item[(i)] Let \(x\) and \(y\) be in \(\mathscr I\). If \(x
\leq_{L} y\) then \(c_{y} \leq_{\Gamma(C)} c_{x}\). \item[(ii)]
Let \(\mathcal{C}\) be a cell of \(\Gamma(C)\) and let
\(\mathscr{I}(\mathcal{C}) = \{\,w \in \mathscr{I} \mid c_{w} \in
\mathcal{C}\,\}\). If \(x,\,y\in\mathscr{I}(\mathcal{C})\) then
the interval \([x, y]_L=\{\,z\in W\mid x\leq_L z\leq_L y\,\}\) is
contained in \(\mathscr{I}(\mathcal{C})\).
\end{itemize}
\end{cor}

Inverting Equation (\ref{qpoly1}), we have
\begin{equation*}
c_{w} = b_{w} - \sum_{y < w}qp_{y,w}b_{y},
\end{equation*}
where \(p_{y,w} \in \mathcal{A}^{+}\) are defined recursively by
\begin{equation*}
p_{y,w} = q_{y,w} - \sum_{y<x<w}qp_{y,x}q_{x,w} \quad \text{if \(y
< w\)}.
\end{equation*}
Note that \(\mu_{y,w}\) is the constant term of \(p_{y,w}\).

\section{Subideals}\label{descr}

The main result of this section says (essentially) that a subideal
of a \(W\!\)-graph ideal is a \(W\!\)-graph ideal provided that
its complement is closed. We assume that \(\mathscr I\) is an
ideal in \((W,\leq_L)\) with \(\mathscr I\subseteq\mathscr{I}_0\),
where \(\mathscr{I}_{0}\) is a \(W\!\)-graph ideal with respect to
\(J \subseteq \Pos(\mathscr{I}_{0}\!)\). We adapt the notation of
Section~\ref{section3} by attaching a subscript or superscript
\(0\) to objects associated with the \(W\!\)-graph ideal
\(\mathscr{I}_0\). Thus we write \(\mathscr{S}_{0}\) for the
\(\mathcal{H}\)-module associated with \(\mathscr{I}_0\) and
\((\,b_{z}^{0} \mid z \in \mathscr I_{0}\,)\) for the basis of
\(\mathscr{S}_{0}\) that satisfies the conditions of
Definition~\ref{wgphdetelt}. By Lemma \ref{uniCbasis1} and Theorem
\ref{main-wg} we know that \(\mathscr{S}_{0}\) has a \(W\!\)-graph
basis \(C_{0} = \{c^{0}_{w} \mid w \in \mathscr{I}_{0}\}\) such
that
\begin{align}\label{btoc}
b^{0}_{w} &= c^{0}_{w} + \sum_{\substack{y \in \mathscr{I}_{0}\\y
<
w}}qq^{0}_{y,w}c^{0}_{y}\\[-12pt]
\intertext{and}\label{ctob} c^{0}_{w} &= b^{0}_{w} -
\sum_{\substack{y \in \mathscr{I}_{0}\\y <
w}}qp^{0}_{y,w}b^{0}_{y}
\end{align}
where the polynomials \(p^{0}_{y,w},\,q^{0}_{y,w} \in
\mathcal{A}^{+}\) are defined whenever~\(y<w\) and are related by
\begin{equation}\label{relpandq}
p^{0}_{y,w} = q^{0}_{y,w} - \sum_{y<x<w} qp^{0}_{y,x}q^{0}_{x,w}.
\end{equation}
Let \(\mu^{0}_{y,w}\) be the constant term of \(q^{0}_{y,w}\) (or,
equivalently, of \(p^{0}_{y,w}\)), so that, by
Theorem~\ref{main-wg}, the triple \(\Gamma(C_{0}) = (C_{0}, \mu,
\tau)\) is a \(W\!\)-graph, where the functions \(\mu\) and
\(\tau\) are given by
\begin{equation*}
\mu(c^{0}_{y},c^{0}_{w})    = \begin{cases}
                               \mu^{0}_{y,w}   &\quad \text{if \(y < w\)}\\
                               \mu^{0}_{w,y}   &\quad \text{if \(w < y\)}
                               \end{cases}
\end{equation*}
and \(\tau(c^{0}_{w}) = \D_{J}(\mathscr{I}_{0},w) =
\SD(\mathscr{I}_{0}, w) \cup \WD_{J}(\mathscr{I}_{0},w)\), for
all~\(y,\,w\in\mathscr{I}_{0}\!\).

Assuming that \(\mathscr I\) is a sub-ideal of \((\mathscr{I}_{0},
\leq_{L})\), let \(\mathscr I' = \mathscr{I}_{0} \setminus
\mathscr I\). Throughout this section, we assume that the set \(C'
= \{c^{0}_{w} \mid w \in \mathscr I'\}\) spans an
\(\mathcal{H}\)-submodule of~\(\mathscr{S}_{0}\). Recall that this
is equivalent to saying that \(C'\) is a closed subset of
\(C_{0}\). By this assumption, the \(\mathcal H\)-module
\(\mathscr{S}' = \mathcal{A}C'\) is obtained from a \(W\!\)-graph,
namely the subgraph of \(\Gamma(C_{0})\) with \(C'\) as its vertex
set and with \(\tau\) invariants and edge weights inherited from
\(\Gamma(C_{0})\). Moreover, the subgraph of \(\Gamma(C_0)\) with
vertex set \(\{\,c_{w}^0 \mid w \in \mathscr I\,\}=C_0\setminus
C'\) and \(\tau\) invariants and edge weights inherited from
\(\Gamma(C_{0})\) is also a \(W\!\)-graph, and the corresponding
\(\mathcal{H}\)-module \(\mathscr{S}\) is isomorphic to
\(\mathscr{S}_{0}/\mathscr{S}'\). We define \(f\colon
\mathscr{S}_{0} \to \mathscr{S}\) to be the homomorphism with
kernel \(\mathscr{S}'\), and define \(c_{w} = f(c^{0}_{w})\) for
all \(w \in \mathscr I\), so that \(C=\{\,c_w\mid w\in\mathscr
I\,\}\) is the \(W\!\)-graph basis of \(\mathscr S\).

Observe that \(J \subseteq \Pos(\mathscr{I}_{0}) \subseteq
\Pos(\mathscr{I})\), and so it makes sense to ask whether
\(\mathscr I\) is a \(W\!\)-graph ideal with respect to~\(J\).
Note that the definitions immmediately imply that
\(\SD(\mathscr{I}_0,w)=\SD(\mathscr{I},w)\) and
\(\WD_J(\mathscr{I}_0,w)=\WD_J(\mathscr{I},w)\), and that
\(\WA_J(\mathscr{I}_0,w)\subseteq\WA_J(\mathscr{I},w)\), for all
\(w\in \mathscr I\). Since there may exist an \(s\in S\) with
\(sw\in\mathscr{I}_0\setminus\mathscr{I}\), it is not necessarily
the case that \(\WA_J(\mathscr{I},w)=\WA_J(\mathscr{I}_0,w)\),

For each \(w\in\mathscr I\) we define \(b_w=T_wc_1\), and we put
\(B=\{\,b_w\mid w\in\mathscr I\,\}\).
\begin{lem}\label{basis-B}
The \(\mathcal{H}\)\,-module \(\mathscr{S}\) is
\(\mathcal{A}\)\,-free with \(\mathcal{A}\)\,-basis \(B\).
\end{lem}
\begin{proof}
Let us compute \(f(b^{0}_{w})\) for each \(w \in \mathscr I\). We
have
\begin{equation*}
f(b^{0}_{w}) = f(T_{w}b^{0}_{1}) = T_{w}f(b^{0}_{1}) =
T_{w}f(c^{0}_{1}) = T_{w}c_{1} = T_{w}b_{1} = b_{w}
\end{equation*}
for each \(w \in \mathscr I\). It now follows from Equation
(\ref{btoc}) that for all \(w \in \mathscr I\)
\begin{align*}
b_{w} &= f(b^{0}_{w})\\
      &= f(c^{0}_{w} + \sum_{\substack{y \in \mathscr I_{0}\\y <
      w}}qq^{0}_{y,w}c^{0}_{y})\\
      &= f(c^{0}_{w}) + \sum_{\substack{y \in \mathscr I_{0}\\y <
      w}}qq^{0}_{y,w}f(c^{0}_{y})\\
      &= c_{w} +  \sum_{\substack{y \in \mathscr I\\y <
      w}}qq^{0}_{y,w}c_{y}
\end{align*}
since \(f(c^{0}_{y})=0\) for all \(y \in
\mathscr{I}_0\setminus\mathscr I\). Thus, for each \(w \in
\mathscr I\)
\begin{equation}\label{btocinI}
b_{w} = c_{w} +  \sum_{\substack{y \in \mathscr I\\y <
      w}}qq_{y,w}c_{y},
\end{equation}
where we have defined \(q_{y,w} = q^{0}_{y,w}\) whenever \(y, w
\in \mathscr I\) and \(y < w\). Now since \(C\) is an
\(\mathcal{A}\)-basis for \(\mathscr{S}\), we see that \(B\) is
also an \(\mathcal{A}\)-basis for \(\mathscr{S}\), as claimed.
\end{proof}
\begin{lem}\label{genGarnir}
For each \(w \in \mathscr I'\),
\begin{equation}\label{garnirEqu}
f(b^{0}_{w}) = \sum_{\substack{y \in \mathscr I\\y < w}}
r_{y,w}b_{y} \quad \text{for some \(r_{y,w} \in
q\mathcal{A}^{+}\)}.
\end{equation}
\end{lem}
\begin{proof}
We have, for each \(w \in \mathscr I'\),
\begin{align}
0 = f(c^{0}_{w}) &= f(b^{0}_{w} - \sum_{\substack{y \in \mathscr
I_{0}\\y < w}} qp^{0}_{y,w}b^{0}_{y})\nonumber\\
                 &= f(b^{0}_{w}) - \sum_{\substack{y \in \mathscr I_{0}\\y <
w}} qp^{0}_{y,w}f(b^{0}_{y})\nonumber\\
                 &= f(b^{0}_{w}) - \sum_{\substack{y \in \mathscr I'\\y <
w}} qp^{0}_{y,w}f(b^{0}_{y}) - \sum_{\substack{y \in \mathscr I\\y
< w}} qp^{0}_{y,w}b_{y}\label{genGarEqu}.
\end{align}
Now by rearranging terms, Equation (\ref{genGarEqu}) becomes
\begin{equation}\label{genGarEquRearr}
f(b^{0}_{w}) = \sum_{\substack{y \in \mathscr I'\\y < w}}
qp^{0}_{y,w}f(b^{0}_{y}) + \sum_{\substack{y \in \mathscr I\\y <
w}} qp^{0}_{y,w}b_{y}
\end{equation}
for each \(w \in \mathscr I'\). We now use induction on \(l(w)\)
to show that Equation (\ref{garnirEqu}) holds for all \(w \in
\mathscr I'\).

Suppose first that \(w\) is of minimal length subject to \(w \in
\mathscr I'\). Equation (\ref{genGarEquRearr}) gives
\begin{equation*}
f(b^{0}_{w}) = \sum_{\substack{y \in \mathscr I\\y < w}}
qp^{0}_{y,w}b_{y}
\end{equation*}
since minimality of \(w\) implies that the set \(\{\,y \in
\mathscr I' \mid y < w\,\}\) is empty. Thus \( f(b^{0}_{w}) =
\sum_{\substack{y \in \mathscr I\\y < w}} r_{y,w}b_{y} \) where
\(r_{y,w} = qp^{0}_{y,w} \in q\mathcal{A}^{+}\), as required.

Now let \(w\in\mathscr I'\) be arbitrary and assume that the
result holds for all \(y \in \mathscr I'\) such that \(l(y) <
l(w)\); that is, assume that
\begin{equation}\label{genGarAss}
f(b^{0}_{y}) = \sum_{\substack{x \in \mathscr I\\x < y}}
r_{x,y}b_{x} \quad \text{for some $r_{x,y} \in q\mathcal{A}^{+}$}.
\end{equation}
Equation (\ref{genGarEquRearr}) and Equation (\ref{genGarAss})
give
\begin{align*}
f(b^{0}_{w}) &= \sum_{\substack{y \in \mathscr I'\\ y < w}}
qp^{0}_{y,w}(\sum_{\substack{x \in \mathscr I\\ x < y}}
r_{x,y}b_{x})
+ \sum_{\substack{y \in \mathscr I\\y < w}} qp^{0}_{y,w}b_{y}\\
    &= \sum_{\substack{x \in \mathscr I\\x<w}}\Bigl(\sum_{\substack{y \in \mathscr
    I'\\ x<y<w}}qp^{0}_{y,w}r_{x,y}\Bigr)b_{x} + \sum_{\substack{x \in \mathscr I\\ x < w}} qp^{0}_{x,w}b_{x}\\
    &= \sum_{\substack{x \in \mathscr I\\ x<w}}\Bigl(qp^{0}_{x,w} + \sum_{\substack{y \in \mathscr
    I'\\ x<y<w}}qp^{0}_{y,w}r_{x,y}\Bigr)b_{x}.
\end{align*}
It follows that
\begin{equation*}
f(b^{0}_{w}) = \sum_{\substack{y \in \mathscr I\\ y < w}}
r_{y,w}b_{y},
\end{equation*}
where \(r_{y,w} = qp^{0}_{y,w} + \sum_{\substack{x \in \mathscr
             I'\\ y<x<w}}qp^{0}_{x,w}r_{y,x} \in
q\mathcal{A}^{+}\), and we are done.
\end{proof}
\begin{rem}\label{garnirRem}
In the expression for \(r_{y,w}\) in the proof above, we see that
if \(y = sw < w\) then \(r_{y,w} = q\), since \(\{\,x\mid
y<x<w\,\} = \emptyset\) and \(p^{0}_{y,w} = q^{0}_{y,w} = 1\).
\end{rem}

For future reference, we state the formula for the coefficients
\(r_{y,w}\) in the proof above in the following corollary (minding
Remark \ref{garnirRem}).
\begin{cor}\label{rywpyw}
For each \(w \in \mathscr I'\) and \(y \in \mathscr I\), the
coefficients appearing in Equation~(\ref{garnirEqu}) are given by
\begin{equation}\label{ryw}
r_{y,w} =  qp^{0}_{y,w}+\sum_{\substack{x \in \mathscr
             I'\\y<x<w}}qp^{0}_{x,w}r_{y,x} \in q\mathcal{A}^{+}.
\end{equation}
In particular, \(r_{y,w} = q\) if \(y = sw < w\).
\end{cor}

We now prove the first main result of the paper.
\begin{thr}\label{wgdetset}
Let \(\mathscr{I}_{0}\) be a \(W\!\)-graph ideal with respect to
\(J \subseteq \Pos(\mathscr{I}_{0})\) and let \(C_{0} =
\{\,c^{0}_{w} \mid w \in \mathscr{I}_{0}\,\}\) be the
\(W\!\)-graph basis of the module
\(\mathscr{S}_0=\mathscr{S}(\mathscr{I}_{0},J)\). Suppose that
\(\mathscr I\) is a sub-ideal of \(\mathscr{I}_{0}\) such that
\(\{\,c^{0}_{w} \mid w \in \mathscr I_{0} \setminus \mathscr
I\,\}\) is a closed subset of \(C_{0}\). Then \(\mathscr I\) is a
\(W\!\)-graph ideal with respect to \(J\). Moreover, the
corresponding \(W\!\)-graph is isomorphic to the  full subgraph of
\(\Gamma(\mathscr{I}_0)\) on the vertex set \(\{\,c_w^0\mid
w\in\mathscr I\,\}\subseteq C_0\), with \(\tau\) and \(\mu\)
functions inherited from \(\Gamma(\mathscr{I}_{0})\).
\end{thr}
\begin{proof}
We need to verify that the ideal \(\mathscr I\) satisfies the
hypotheses required in Definition \ref{wgphdetelt}. All we need to
show is that the \(\mathcal{H}\)-module \(\mathscr S\) as
constructed above satisfies the required conditions. By Lemma
\ref{basis-B}, \(\mathscr S\) is \(\mathcal{A}\)-free with a free
\(\mathcal A\)-basis given by \(B = \{\,b_{w} \mid w \in \mathscr
I\,\}\), where \(b_{w} = f(b^{0}_{w})\) for each \(w \in \mathscr
I\). (Recall that \(f\) is natural homomorphism from \(\mathscr
S_{0}\) onto \(\mathscr S\).) To complete the verification of
Condition (i) in Definition \ref{wgphdetelt}, we proceed to work
out how the generators \(T_{s}\) act on the basis
elements~\(b_{w}\).

Let \(w \in \mathscr I\) and let \(s \in S\). If \(s \in
\SA(\mathscr{I},w)\) then \(w<sw\in\mathscr I\subseteq\mathscr
I_0\), whence \(s \in \SA(\mathscr{I}_{0},w)\), and Equation
(\ref{S_0action}) for \(\mathscr S_0\) gives
\(T_{s}b^{0}_{w}=b^{0}_{sw}\). So
\begin{equation*}
T_{s}b_{w} = T_{s}f(b^{0}_{w}) = f(T_{s}b^{0}_{w}) = f(b^{0}_{sw})
= b_{sw}
\end{equation*}
in accordance with the requirements of
Definition~\ref{wgphdetelt}~(i). If \(s \in \SD(\mathscr{I},w)\)
then \(w>sw\), whence \(s\in\SD(\mathscr{I}_{0},w)\), and
\(T_{s}b^{0}_{w}=b^{0}_{sw} + (q-q^{-1})b^{0}_{w}\) by
Equation~(\ref{S_0action}) for \(\mathscr S_0\). So
\begin{equation*}
T_{s}b_{w} = T_{s}f(b^{0}_{w}) = f(T_{s}b^{0}_{w}) = f(b^{0}_{sw}
+ (q-q^{-1})b^{0}_{w}) = b_{sw} + (q-q^{-1})b_{w},
\end{equation*}
again in accordance with the requirements of
Definition~\ref{wgphdetelt}~(i). If \(s \in
\WD_{J}(\mathscr{I},w)\) then \(sw=wt\) for some \(t\in J\),
whence \(s\in\WD_{J}(\mathscr{I}_{0},w)\), and Equation
(\ref{S_0action}) for \(\mathscr S_0\) gives
\(T_{s}b^{0}_{w}=-q^{-1}b^{0}_{sw}\). So
\begin{equation*}
T_{s}b_{w} = T_{s}f(b^{0}_{w}) = f(T_{s}b^{0}_{w}) =
f(-q^{-1}b^{0}_{w}) = -q^{-1}b_{w},
\end{equation*}
and again the requirements of Definition~\ref{wgphdetelt}~(i) are
satisfied.

Finally, suppose that \(s \in \WA_{J}(\mathscr{I},w)\), so that
\(w<sw\notin\mathscr I\). It follows that either \(s \in
\SA_{J}(\mathscr{I}_{0},w)\) or
\(s\in\WA_{J}(\mathscr{I}_{0},w)\), depending on whether
\(sw\in\mathscr I_0\setminus\mathscr I\) or \(sw\notin\mathscr
I_0\). In the former case Equation (\ref{S_0action}) for
\(\mathscr S_0\) gives \(T_{s}b^{0}_{w}=b^{0}_{sw}\), and so
\begin{equation}\label{actsinSAGen}
T_{s}b_{w} = T_{s}f(b^{0}_{w}) = f(T_{s}b^{0}_{w}) = f(b^{0}_{sw})
= qb_{w} - \sum_{\substack{y \in \mathscr I\\y
<sw}}r^{s}_{y,w}b_{y},
\end{equation}
where \(r^{s}_{y,w} = -r_{y,sw} \in q\mathcal{A}^{+}\), by Lemma
\ref{genGarnir} and Corollary \ref{rywpyw}. On the other hand, if
\(s \in \WA_{J}(\mathscr{I}_{0},w)\), then, by Equation
(\ref{S_0action}) for \(\mathscr S_0\) and Lemma \ref{genGarnir},
\begin{align*}
T_{s}b_{w} &= T_{s}f(b^{0}_{w}) = f(T_{s}b^{0}_{w}) = f(qb^{0}_{w} - \sum_{\substack{y \in \mathscr{I}_{0}\\y<sw}}r^{0s}_{y,w}b^{0}_{y})\\
&= qb_{w} - \sum_{\substack{y \in \mathscr{I}\\y<sw}}r^{0s}_{y,w}b_{y} - \sum_{\substack{y \in \mathscr{I}'\\y<sw}}r^{0s}_{y,w}f(b^{0}_{y})\\
&= qb_{w} - \sum_{\substack{y \in \mathscr{I}\\y<sw}}r^{0s}_{y,w}b_{y} - \sum_{\substack{y \in \mathscr{I}'\\y<sw}}r^{0s}_{y,w}\Bigl(\sum_{\substack{x \in \mathscr{I}\\x<y}}r_{x,y}b_{x}\Bigr)\\
&= qb_{w} - \sum_{\substack{y \in \mathscr{I}\\y<sw}}r^{0s}_{y,w}b_{y} - \sum_{\substack{x \in \mathscr{I}\\x<sw}}\Bigl(\sum_{\substack{y \in \mathscr{I}'\\x<y<sw}}r^{0s}_{y,w}r_{x,y}\Bigr)b_{x}\\
&= qb_{w} - \sum_{\substack{y \in
\mathscr{I}\\y<sw}}\Bigl(r^{0s}_{y,w}+\sum_{\substack{x \in
\mathscr{I}'\\y<x<sw}}r^{0s}_{x,w}r_{y,x}\Bigr)b_{y}.
\end{align*}
Thus we have shown that
\begin{equation}\label{actsinWAGen}
T_{s}b_{w}= qb_{w} - \sum_{\substack{y \in \mathscr
I\\y<sw}}r^{s}_{y,w}b_{y}
\end{equation}
where \(r^{s}_{y,w} = r^{0s}_{y,w}+\sum r^{0s}_{x,w}r_{y,x}\), and
\(r^{s}_{y,w}\in q\mathcal{A}^{+}\) by Definition \ref{wgphdetelt}
and Corollary~\ref{rywpyw}. Hence in either case the requirements
of Definition~\ref{wgphdetelt}~(i) are satisfied.

The second assertion of the theorem is obviously satisfied, by the
way we defined the \(\mathcal{H}\)-module~\(\mathscr S\). This
also ensures that Condition (ii) of Definition \ref{wgphdetelt}
holds, since, as we observed in Section~\ref{Cox}, every module
arising from a \(W\!\)-graph admits a semilinear involution
\(\alpha\mapsto\overline\alpha\) that fixes the elements of the
\(W\!\)-graph basis and satisfies
\(\overline{h\alpha}=\overline{h}\overline\alpha\) for all \(h \in
\mathcal H\).
\end{proof}

Let \(\mathscr{I}\) be a \(W\!\)-graph ideal with respect to \(J
\subseteq \Pos(\mathscr{I})\) and let \(C = \{\,c_{w} \mid w \in
\mathscr I\,\}\) be the corresponding \(W\!\)-graph basis of the
module \(\mathscr S(\mathscr I,J)\). Let \(\mathscr{C}\) be the
set of cells of \(\Gamma = \Gamma(C)\). We have the following
result.

\begin{lem}\label{mincell}
Let \(\mathcal{C}_1\in\mathscr C\) be the cell that contains
\(c_{1}\). Then \(\mathcal{C}_1\) is the unique maximal element of
\((\mathscr{C},\leq_{\Gamma})\).
\end{lem}
\begin{proof}
The result follows readily from Corollary \ref{edgeup}~(i) and the
fact that \(1 \leq_{L} w\) for all \(w \in \mathscr{I}\).
\end{proof}

For every \(D\subseteq C\) define \(\mathscr I_D=\{\,w\in\mathscr
I\mid c_w\in D\,\}\), and for each cell \(\mathcal C\in\mathscr
C\) define \(\overline{\mathcal C}\) to be the union of all
\(\mathcal D\in\mathscr C\) such that \(\mathcal
C\le_\Gamma\mathcal D\). Note that Lemma~\ref{mincell} tells us
that \(\mathcal C_1\subseteq\overline{\mathcal C}\) and hence that
\(1\in\mathscr I_{\,\overline{\mathcal C}}\).

\begin{lem}\label{leftidealIk}
For each \(\mathcal C\in\mathscr C\) the sets
\(\mathscr{I}_{\,\overline{\mathcal C}}\) and
\(\mathscr{I}_{\,\overline{\mathcal C}\setminus\mathcal C}\) are
ideals of \((W,\leq_{L})\).
\end{lem}
\begin{proof}
To show that \(\mathscr{I}_{\,\overline{\mathcal C}}\) is an ideal
it suffices to show that if \(w \in
\mathscr{I}_{\,\overline{\mathcal C}}\) and \(s\in S\) with \(sw
<_{L} w\), then \(sw \in \mathscr{I}_{\,\overline{\mathcal C}}\).
Let \(\mathcal{D}\) and \(\mathcal{D}'\) be the cells that contain
\(c_{w}\) and \(c_{sw}\) respectively. Since \(sw <_{L} w\) it
follows from Corollary \ref{edgeup}~(i) that \(\mathcal{D}'
\geq_{\Gamma} \mathcal{D}\). But \(\mathcal{D}\geq_{\Gamma}
\mathcal{C}\) since \(w \in \mathscr{I}_{\,\overline{\mathcal
C}}\); so \(\mathcal{D}'\geq_{\Gamma} \mathcal{C}\). Hence
\(c_{sw}\in \overline{\mathcal C}\), so that \(sw \in
\mathscr{I}_{\,\overline{\mathcal C}}\), as desired.

The proof of the other part is similar.
\end{proof}

\begin{lem}\label{backclosed}
Let \(\mathcal C \in\mathscr C\). Then \(\overline{\mathcal C}{}'
= C \setminus \overline{\mathcal C}\) and \(\overline{\mathcal
C}{}'\cup\mathcal C\) are closed subsets of~\(C\).
\end{lem}
\begin{proof}
To show that \(\overline{\mathcal C}{}'\) is closed it is
sufficient to show that whenever \(c_{w} \in \overline{\mathcal
C}{}'\) and \(c_{y} \in C\) is such that \(c_{y} \le_\Gamma
c_{w}\), then \(c_{y} \in \overline{\mathcal C}{}'\). Given such
elements \(c_w\) and \(c_y\), let \(\mathcal{Y}\) and
\(\mathcal{W}\) be the cells that contain \(c_{y}\) and \(c_w\).
Then \(\mathcal{W}\not\ge_\Gamma\mathcal{C}\) since  \(c_{w}
\notin \overline{\mathcal C}\), and since
\(\mathcal{W}\ge_\Gamma\mathcal{Y}\) it follows that
\(\mathcal{Y}\not\ge_\Gamma\mathcal{C}\), whence \(c_{y}\notin
\overline{\mathcal C}\), as required. A similar argument proves
that \(\overline{\mathcal C}{}'\cup\mathcal C\) is also closed.
\end{proof}

\begin{cor}\label{leftidealcells}
For each \(\mathcal C \in\mathscr C\) the sets
\(\mathscr{I}_{\,\overline{\mathcal C}}\) and
\(\mathscr{I}_{\,\overline{\mathcal C}\setminus\mathcal C}\) are
\(W\!\)-graph ideals with respect to~\(J\). The associated
\(W\!\)-graphs are the corresponding full subgraphs of \(\Gamma\),
with \(\tau\) and \(\mu\) inherited from~\(\Gamma\).
\end{cor}

\begin{proof}
Since Lemma \ref{leftidealIk} shows that
\(\mathscr{I}_{\,\overline{\mathcal C}}\) and
\(\mathscr{I}_{\,\overline{\mathcal C}\setminus\mathcal C}\) are
subideals of the \(W\!\)-graph ideal~\(\mathscr I\!\), and
Lemma~\ref{backclosed} shows that the complements of
\(\overline{\mathcal C}\) and \(\overline{\mathcal
C}\setminus\mathcal C\) are closed subsets of~\(C\), the result
follows immediately from Theorem \ref{wgdetset}.
\end{proof}

\begin{rem}\label{filtration}
In the above situation, the closed subsets \(\overline{\mathcal
C}{}'\) and \(\overline{\mathcal C}{}'\cup\mathcal C\) of \(C\)
span \(\mathcal H\)-submodules \(M_{\Gamma(\overline{\mathcal
C}{}')}\) and \(M_{\Gamma(\overline{\mathcal C}{}'\cup\mathcal
C)}\) of \(M_\Gamma=\mathscr S(\mathscr I,J)\). Furthermore, the
factor module \(M_{\Gamma(\overline{\mathcal C}{}'\cup\mathcal
C)}/M_{\Gamma(\overline{\mathcal C}{}')}\) is isomorphic to the
\(\mathcal H\)-module determined by the cell~\(\mathcal C\), which
in turn is isomorphic to the kernel of the natural homomorphism
\(f\colon \mathscr S(\mathscr{I}_{\,\overline{\mathcal C}},J)\to
\mathscr S(\mathscr{I}_{\,\overline{\mathcal C}\setminus\mathcal
C},J)\).
\end{rem}

For later reference, we record the following special case of
Corollary \ref{leftidealcells}, obtained by setting \(\mathcal C =
\mathcal C_1\).

\begin{lem}\label{bottomcellofI0}
The set \(\mathscr{I}_{1}=\mathscr{I}_{\mathcal C_1}\) is a
\(W\!\)-graph ideal, and the corresponding \(W\!\)-graph is
exactly that of the maximal cell of \(\Gamma\).
\end{lem}

\section{Cells in the ideal of minimal coset representatives}\label{dcrcells}

Let \((W,S)\) be a Coxeter system, and let \(\mathcal{H} =
\mathcal{H}(W)\) be the associated Hecke algebra. Let \(J\) be an
arbitrary subset of~\(S\) and \(D_{J}\) the set of distinguished
left coset representatives for~\(W_{J}\). It is easily shown that
if \(u\in W\) is a suffix of some element of \(D_J\) then \(u\) is
also in~\(D_J\); so \(\mathscr I=D_J\) is an ideal of
\((W,\le_L)\). Clearly \(\Pos(\mathscr I)=J\). In \cite[Section
8]{howvan:wgraphDetSets}, it is shown that \(\mathscr I\) is a
\(W\!\)-graph ideal with respect to \(\emptyset\) and also with
respect to~\(J\). Here we consider the latter case only. We
briefly review the main facts, referring the reader to
\cite{howvan:wgraphDetSets} for the full details.

Let \(\mathcal{H}_J\) be the Hecke algebra associated with the
Coxeter system \((W_{J},J)\). Let \(\mathcal{A}_\phi\) be
\(\mathcal{A}\) made into an \(\mathcal{H}_J\)-module via the
homomorphism \(\phi\colon\mathcal{H}_J\to\mathcal{A}\) defined by
\(\phi(T_u)=(-q)^{-l(u)}\) for all \(u\in W_J\), and let
\(\mathscr{S}_{\!\phi}=\mathcal{H}\otimes_{\mathcal{H}_J}\mathcal{A}_\phi\),
the \(\mathcal{H}\)-module induced from~\(\mathcal{A}_\phi\) (so
that \(\mathscr{S}_{\!\phi}\) is essentially the module \(M^J\) of
\cite{deo:paraKL} in the case \(u=-1\)). Then
\(\mathscr{S}_{\!\phi}\) is \(\mathcal{A}\)-free with
\(\mathcal{A}\)-basis \(B_{J} = (\,b^{J}_{w} \mid w \in
D_{J}\,)\), where \(b^{J}_{w}=T_{w}\otimes 1\) for each \(w \in
D_{J}\). All the conditions in Definition \ref{wgphdetelt} are
satisfied, and so \(\mathscr{S}_{\phi}\) has a \(W\!\)-graph basis
\(C_{J} = (\,c^{J}_{w} \mid w \in D_{J}\,)\) such that
\(c^{J}_{w}=b^{J}_{w}-\sum_{y<w} qp^{J}_{y,w}b^{J}_{y}\) for all
\(w\in D_{J}\), where the polynomials~\(p^{J}_{y,w}\) are given by
the formulas in Section~\ref{wgdetset} above. Note that in the
special case \(J=\emptyset\) the module \(\mathscr{S}_{\!\phi}\)
is isomorphic to the left regular module~\(\mathcal H\), and the
\(W\!\)-graph basis is \(C_\emptyset=(\,c_w\mid w\in W\,)\), the
Kazhdan-Lusztig basis of~\(\mathcal H\). In this case the
\(W\!\)-graph~\(\Gamma(C_{J})\) becomes the regular
Kazhdan-Lusztig \(W\!\)-graph~\(\Gamma(W)\), and
\(c_{w}=T_{w}-\sum_{y<w} qp_{y,w}T_{y}\) for all \(w\in W\); see
\cite[Proposition~8.2]{howvan:wgraphDetSets}.

We shall show that if \(W_J\) is finite then the
\(W\!\)-graph~\(\Gamma(C_{J})\) is isomorphic to the \(W\!\)-graph
of a certain union of left cells in~\(\Gamma(W)\).
\begin{prop}\label{deo2}If \(J\subseteq S\) and \(W_J\) is finite then then the
polynomials \(p^{J}_{y,w}\) and \(p_{y,w}\) defined above are
related via the formula \( p^{J}_{y,w} = p_{yw_{J},ww_{J}}, \)
where \(w_{J}\) is the longest element in \(W_{J}\).
\end{prop}
\begin{proof}
In view of the relationship between our polynomials
\(p^{J}_{y,w}\) and Deodhar's polynomials \(P^{J}_{y,w}\) (see
\cite[Proposition~8.4]{howvan:wgraphDetSets}), and the
relationship between our polynomials \(p_{y,w}\) and the
Kazhdan-Lusztig polynomials \(P_{y,w}\) (see
\cite[Proposition~8.2]{howvan:wgraphDetSets}), this result is
immediate from \cite[Proposition~3.4]{deo:paraKL}.
\end{proof}
For each \(w\in W\), define \(L(w)=\{\,s\in S\mid sw<w\,\}\). Note
that this is the \(\tau\)-invariant of the vertex \(c_w\)
of~\(\Gamma(W)\).
\begin{lem}\label{descentcomp}
If \(W_J\) is finite and \(w \in D_{J}\), then \(L(ww_{J}) =
\D_{J}(\mathscr I,w)\), where \(\mathscr I=D_J\) (as above).
\end{lem}
\begin{proof}
Suppose first that \(s \in L(ww_{J})\), so that \(sww_{J} <
ww_{J}\). If \(sw < w\) then \(s \in \SD_J(w)\subseteq \D_J(w)\).
On the other hand, if \(sw > w\) then \(sw \notin D_{J}=\mathscr
I\), since otherwise we would have \(l(sww_{J}) = l(sw) + l(w_{J})
> l(w) + l(w_{J})=l(ww_{J})\). So in this case \(s \in
\WD_{J}(w)\subseteq \D_J(w)\), and we conclude that
\(L(ww_{J})\subseteq D_{J}(w)\).

Suppose conversely that \(s \in \D_{J}(w)\). If \(s \in \SD(w)\)
then \(sw < w\), and it follows that \(l(sww_{J}) = l(sw) +
l(w_{J}) < l(w) + l(w_{J}) = l(ww_{J})\), so that \(s \in
L(ww_{J})\). If \(s \in \WD_{J}(w)\) then \(sw\notin \mathscr I =
D_{J}\), but since \(w\in D_{J}\) it follows from Lemma~2.1~(iii)
of \cite{deo:paraKL} that \(sw=ws'\) for some \(s' \in J\). Thus
\(l(sww_{J}) = l(ws'w_{j}) = l(w) + l(s'w_{J})\), since \(w\in
D_J\) and \(s'w_{J}\in W_{J}\). But \(l(s'w_{J}) < l(w_{J})\),
since \(w_{J}\) is the longest element in \(W_{J}\), and so \(l(w)
+ l(s'w_{J})< l(w) + l(w_{J}) = l(ww_{J})\). Hence \(s\in
L(ww_{J})\), and we conclude that \(D_{J}(w)\subseteq L(ww_{J})\),
as required.
\end{proof}
The following result is immediate from
\cite[Lemma~2.8]{geck:klMurphy}.
\begin{lem}\label{geck1}
The set \(D_{J}w_{J}\) is a union of Kazhdan-Lusztig left cells:
we have
\begin{equation*}
D_{J}w_{J} = \{w \in W \mid w \le_{\Gamma} w_{J}\}.
\end{equation*}
Furthermore, \(\mathcal{A}C_{J} = M_{\Gamma(C_J)} \cong
\mathcal{H}c_{w_{J}} \cong \mathscr{S}_{\phi}\) as left
\(\mathcal{H}\)-modules.
\end{lem}
More explicitly, the \(\mathcal{H}\)-module isomorphism
\(M_{\Gamma(C_J)} \cong \mathcal{H}c_{w_{J}}\) derives from an
isomorphism of \(W\!\)-graphs. Lemma~\ref{descentcomp} and
Proposition~\ref{deo2} show that the mapping \(c_{w}^{J}\mapsto
c_{ww_{J}}\) from \(C_{J}\) to \(C_{\emptyset}\) induces an
isomorphism of the \(W\!\)-graph \(\Gamma(C_J)\) (obtained from
\(D_J\) considered as a \(W\!\)-graph ideal with respect to \(J\))
with the full subgraph of \(\Gamma(W)\) corresponding to the set
\(D_{J}w_{J}\). The mapping preserves edge-weights and
\(\tau\)-invariants, and hence preserves cells and the partial
order on cells.

As an immediate consequence of Lemma \ref{bottomcellofI0} and the
above remarks, we obtain the second main result of this paper.
\begin{thr}\label{wgdetsetcell}
Let \(J\subseteq S\) be such that \(W_J\) is finite, and let
\(\mathcal{C}\) be the Kazhdan-Lusztig left cell that contains
\(w_{J}\).  Then \(\mathcal{C} = \mathscr{I\!}_1w_{J}\), where
\(\mathscr{I\!}_1\subseteq D_{J}\) is a \(W\!\)-graph ideal with
respect to~\(J\), and \(\mathcal{C}_1=\{\,c^{J}_{w}\mid w\in
\mathscr{I\!}_1\,\}\) is the maximal cell of~\(\,\Gamma(C_{J})\).
\end{thr}

The cell \(\mathcal C\) in Theorem~\ref{wgdetsetcell} is the
maximal cell in \(D_{J}w_{J}\), and the theorem tells us that
\(\mathcal Hc_{w_{J}}/\mathcal A(D_{J}w_{J}\setminus\mathcal C)
\cong \mathcal{A}C_J/\mathcal A(C_J\setminus\mathcal{C}_1)\cong
\mathscr S(\mathscr{I\!}_1,J)\) as \(\mathcal H\)-modules.

\section{\(W\!\)-graphs for left cells in type \(A\)}

Let \(W_{n}\) be the Coxeter group of type \(A_{n-1}\), which we
identify with the symmetric group on \([1,n]\), the set of
integers from 1 to~\(n\), by identifying the simple reflections
\(s_1,\,s_2,\,\ldots,\,s_{n-1}\) in \(W_n\) with the
transpositions \((1,2),\,(3,4),\,\ldots,\,(n-1,n)\)
(respectively). We use a left-operator convention for
permutations, writing \(wi\) for the action of \(w\in W_n\) on
\(i\in[1,n]\).

Since the principal objective of this section is to prove
Proposition~6.3 of~\cite{howvan:wgraphDetSets}, we start by
reviewing the conventions and terminology of that paper.

A sequence of nonnegative integers \(\lambda =
(\lambda_{1},\lambda_{2} \ldots, \lambda_{k})\) is called a
partition of \(n\) if \(\lambda_{1} + \lambda_{2} + \cdots +
\lambda_{k} = n\) and \(\lambda_{1} \ge \cdots \ge \lambda_{k}\).
We define \(\mathcal P(n)\) to be the set of all partitions
of~\(n\). For each
\(\lambda=(\lambda_{1},\ldots,\lambda_{k})\in\mathcal P(n)\) we
define
\[
[\lambda]=\{\,(i,j)\mid1\leq j\leq\lambda_{i}\text{ and }1\leq
i\leq k\,\},
\]
and refer to this as the Young diagram of~\(\lambda\). Pictorially
\([\lambda]\) is represented by a left-justified array of boxes
with \(\lambda_{i}\) boxes in the \(i\)-{th} row; the pair
\((i,j)\in[\lambda]\) corresponds to the \(j\)-{th} box in the
\(i\)-{th} row.

For \(\lambda \in\mathcal P(n)\), define \(\lambda'\) by
\(\lambda_{i}' = \abs{\{\,\lambda_{j} \mid \lambda_{j} \geq
i\,\}}\), and call \(\lambda'\) the partition conjugate to
\(\lambda\). The Young diagram of \(\lambda'\) is the transpose of
the Young diagram of \(\lambda\): the number of boxes in the
\(i\)-th column of \([\lambda']\) equals the number of boxes in
the \(i\)-th row of \([\lambda']\).

If \(\lambda\) is a partition of \(n\) then a
\(\lambda\)\textit{-tableau} is a bijection \(t\colon[\lambda]
\rightarrow [1,n] \). The partition \(\lambda\) is called the
\textit{shape\/} of the tableau \(t\), and we write
\(\lambda=\shape(t)\). For each \(i\in[1,n]\) we define
\(\row_t(i)\) and \(\col_t(i)\) to be the row index and column
index of \(i\)~in~\(t\) (so that
\(t^{-1}(i)=(\row_t(i),\col_t(i))\). We define \(\Tab(\lambda)\)
to be the set of all \(\lambda\)-tableaux, and we let
\(\tab^\lambda\) be the specific \(\lambda\)-tableau given by
\[
\tab^\lambda(i,j)=j+\sum_{h=1}^{i-1}\lambda_h
\]
for all \((i,j)\in[\lambda]\). That is, the numbers
\(1,\,2,\,\dots,\,\lambda_1\) fill the first row of \([\lambda]\)
in order from left to right, then the numbers
\(\lambda_1+1,\,\lambda_1+2,\,\dots,\,\lambda_1+\lambda_2\)
similarly fill the second row, and so on. We also define
\(\tab_{\lambda}\) to be the \(\lambda\)-tableau that is the
transpose of the \(\lambda'\)-tableau \(\tab^{\lambda'}\), where
\(\lambda'\) is the conjugate of~\(\lambda\). Thus in
\(\tab_{\lambda}\) the numbers \(1,\,2,\,\dots,\,\lambda_1'\) fill
the first column in order from top to bottom, then the numbers
\(\lambda_1'+1,\,\lambda_1'+2,\,\dots,\,\lambda_1'+\lambda_2'\)
fill the next column from top to bottom, and so on.

It is clear that for any fixed \(\lambda\in P(n)\) the group
\(W_{n}\) acts on the set of all \(\lambda\)-tableaux, via
\((wt)(i,j) = w(t(i,j))\) for all \((i,j)\in[\lambda]\), for all
\(\lambda\)-tableaux \(t\) and all \(w \in W_{n}\). Moreover, the
map from \(W_n\) to \(\Tab(\lambda)\) defined by \(w\mapsto
w\tab_\lambda\) for all \(w\in W_n\) is bijective. We use this
bijection to transfer the left weak Bruhat and the Bruhat partial
orders from \(W_n\) to \(\Tab(\lambda)\). Thus if \(t_1,\,t_2\)
are arbitrary \(\lambda\)-tableaux  and we write
\(t_1=w_1\tab_\lambda\) and \(t_2=w_2\tab_\lambda\) with
\(w_1,\,w_2\in W_n\), then by definition \(t_1\le t_2\) if and
only if \(w_1\le w_2\), and \(t_1\le_L t_2\) if and only if
\(w_1\le_L w_2\). Similarly, if \(t=w\tab_\lambda\) is an
arbitrary \(\lambda\)-tableau, where \(w\in W_n\), then we define
\(l(t)=l(w)\).

A \(\lambda\)-tableau \(t\), where \(\lambda \in \mathcal P(n)\),
is said to be \textit{column standard\/} if its entries increase
down the columns, that is, if \(t(i,j) < t(i+1,j)\) whenever
\((i,j)\in[\lambda]\) and \((i+1,j)\in[\lambda]\). Similarly,
\(t\) is said to be \textit{row standard\/} if its entries
increase along the rows, that is, if \(t(i,j) < t(i,j+1)\)
whenever \((i,j)\in[\lambda]\) and \((i,j+1)\in[\lambda]\). A
\textit{standard tableau\/} is a tableau that is both column
standard and row standard. We write \(\STD(\lambda)\) for the set
of all standard tableaux for \(\lambda\).

Given \(\lambda\in \mathcal P(n)\) we define \(J_\lambda\) to be
the subset of \(S\) consisting of those simple reflections
\(s_i=(i,i+1)\) such that \(i\) and \(i+1\) lie in the same column
of~\(\tab_\lambda\), and we define \(W_\lambda\) to be the
standard parabolic subgroup of \(W_n\) generated by~\(J_\lambda\).
Thus \(W_\lambda\) is the column stabilizer of~\(\tab_\lambda\).
Moreover, the set of minimal left coset representatives for
\(W_\lambda\) in \(W_n\) is the set
\[
D_\lambda=\{\,d\in W_n\mid di<d(i+1)\text{ whenever \(s_i\in
J_\lambda\)}\,\}
\]
since the condition \(di<d(i+1)\) is equivalent to
\(l(ds_i)>l(d)\). It follows that \(\{\,d\tab_\lambda\mid d\in
D_\lambda\,\}\) is precisely the set of column standard
\(\lambda\)-tableaux.

We have the following result (see for example \cite[Lemma
1.5]{dipjam:heckA}, \cite[Lemma 6.2]{howvan:wgraphDetSets}).
\begin{lem}\label{characterisetstd1}
Let \(\lambda \in\mathcal  P(n)\) and define \(v_\lambda\in W_n\)
by the requirement that \(\tab^\lambda=v_\lambda\tab_\lambda\).
Then \(\STD(\lambda)=\{\,w\tab_\lambda \mid w\le_L v_\lambda\,\} =
\{\,t \in \Tab(\lambda) \mid t\le_L \tab^{\lambda}\,\}\).
\end{lem}

The Robinson-Schensted algorithm associates each \(w \in W_{n}\)
with an ordered pair of standard tableaux of the same shape
\(\lambda\) for some \(\lambda \in \mathcal P(n)\). Moreover, this
gives a bijection from \(W_n\) to the set of all such pairs.
\begin{thr}\label{rs}
The Robinson-Schensted map \(w \mapsto RS(w) = (P(w),Q(w))\) is a
bijection from \(W_{n}\) to \(\{\,(t,u)\in\mathcal P(n)^2\mid
\shape(t)=\shape(u)\,\}\).
\end{thr}

See, for example, \cite[Theorem 3.1.1]{sag:sym}. Details of the
algorithm can also be found (for example) in \cite[Section
3.1]{sag:sym}.

The following lemma, the proof of which relies on the details of
the Robinson-Schensted algorithm, is of crucial importance to us.

\begin{lem}\label{rstlambda}
Let \(\lambda\in\mathcal P(n)\) and let \(w\in W_n\). Then
\(RS(w)=(t,\tab_\lambda)\) for some \(t\in\STD(\lambda)\) if and
only if \(w=vw_{J_\lambda}\) for some \(v\in W\) such that
\(v\tab_\lambda\in\STD(\lambda)\). When these conditions hold,
\(t=v\tab_\lambda\).
\end{lem}
\begin{proof}
Write \(\lambda'\) for the partition conjugate to~\(\lambda\), so
that \(\lambda_j'\) is the number of boxes in the \(j\)-th column
of the Yound diagram~\([\lambda]\). Recall that \(W_{J_\lambda}\)
is the column stabilizer of \(\tab_\lambda\); hence it can be seen
that \(w_{J_\lambda}\tab_\lambda\) is obtained from
\(\tab_\lambda\) by reversing the orders of the numbers in each of
the columns. That is, the numbers \(1,\,2,\,\dots,\,\lambda_1'\)
fill the first column of the tableau \(w_{J_\lambda}\tab_\lambda\)
in order from bottom to top, then the numbers
\(\lambda_1'+1,\,\lambda_1'+2,\,\dots,\,\lambda_1'+\lambda_2'\)
fill the next column from bottom to top, and so on.

Recall that \(RS(w)\) is obtained by applying the row-insertion
and recording process successively to the terms of the sequence
\((w1,w2,\ldots,wn)\). Suppose that \(RS(w)=(t,\tab_\lambda)\), so
that the recording tableau is~\(t_\lambda\). Since the numbers
\(1,\,2,\,\ldots,\,\lambda_1'\) make up the first column
of~\(\tab_\lambda\), the first \(\lambda_1'\) insertions must go
into the first column of the insertion tableau. This means that
\(w1>w2>\cdots>w{\lambda_1'}\), since each successive one of these
bumps the preceding one into the next row, and the result is that
the first column of \(t\) contains the numbers
\(w1,w2,\ldots,w{\lambda_1'}\) in order from bottom to top. In
other words, for \(1\le i\le\lambda_1'\), the position of \(wi\)
in \(t\) is the same as the position of \(i\)
in~\(w_{J_\lambda}\tab_\lambda\). Similarly, the next
\(\lambda_2'\) insertions must go into the second column of~\(t\),
and we conclude that for \(\lambda_1'+1\le
i\le\lambda_1'+\lambda_2'\) the position of \(wi\) in \(t\) is the
same as the position of \(i\) in~\(w_{J_\lambda}\tab_\lambda\).
Clearly the same reasoning applies to all the columns, and it
follows that \(t\) is obtained from \(w_{J_\lambda}\tab_\lambda\)
by replacing \(i\) by \(wi\) for all \(i\in[1,n]\). In other
words, \(t=ww_{J_\lambda}\tab_\lambda\), as required.

Conversely, suppose that \(w=vw_{J_\lambda}\), where
\(t=v\tab_\lambda\) is a standard tableau. Then
\(w(w_{J_\lambda}\tab_\lambda)=t\), and it follows that
\(w1>w2>\cdots>w\lambda_1'\). This in turn implies that the first
\(\lambda_1'\) insertions go into the first column. Similarly the
next \(\lambda_2'\) insertions go into the second column, and so
on, and it follows that the recording tableau is~\(\tab_\lambda\),
as required.
\end{proof}

The following result is well-known. (See \cite[Theorem
1.4]{kazlus:coxhecke} and \cite[Theorem A]{ariki:RSleftcells}.)
\begin{thr}\label{kl}
If \(t\) is a fixed tableau, then the set \(\mathcal{C}(t) = \{\,w
\in W_{n} \mid Q(w) = t\,\}\) is a left cell, and if \(t'\) is
another tableau then \(\mathcal{C}(t')\) is isomorphic to
\(\mathcal{C}(t)\) if and only if \(t'\) and \(t\) are of the same
shape.
\end{thr}

By Theorem~\ref{kl} the set \(\mathcal C_\lambda = \{\,w\in
W_n\mid RS(w)=(t,\tab_\lambda) \text{ for some
\(t\in\STD(\lambda)\)}\,\}\) is a left cell in \(W_n\), and by
Lemma~\ref{rstlambda} it contains \(w_J\), where \(J=J_\lambda\).
Hence it follows from Theorem~\ref{wgdetsetcell} that \(\mathcal
C_\lambda=\mathscr I\!_\lambda w_J\) where \(\mathscr
I\!_\lambda\) is a \(W\!\)-graph ideal with respect to~\(J\), and
\(\mathcal C_1^J=\{\,c^J_w\mid w\in \mathscr I\!_\lambda\,\}\) is
the maximal cell of~\(\Gamma(C_J)\). Furthermore,
Lemma~\ref{rstlambda} also shows that \(\mathscr
I\!_\lambda=\{\,v\in W\mid v\tab_\lambda\in\STD(\lambda)\,\}\),
and by Lemma~\ref{characterisetstd1} this is the ideal of
\((W,\leq_L)\) generated by the element \(v_\lambda\) such that
\(v_\lambda\tab_\lambda=\tab^\lambda\).

Thus we have proved the following results.

\begin{prop}\label{leftcellwj} Let \(\lambda\in\mathcal P(n)\) and let
\(W_J\) be the column stabilizer of~\(\tab_\lambda\). Let
\(v_\lambda\) be the element of~\(W_n\) such that
\(v_\lambda\tab_\lambda=\tab^\lambda\) and let \(\mathscr
I\!_\lambda=\{\,v\in W\mid v\leq_L v_\lambda\,\}\). Then the
Kazhdan-Lusztig left cell that contains \(w_J\) is \(\mathscr
{I}\!_\lambda w_J\).
\end{prop}

\begin{thr}\label{wdetelmA} Let \(\lambda\in\mathcal P(n)\) and let \(W_J\) be the
column stabilizer of~\(\tab_\lambda\). Then the element
\(v_{\lambda}\in W_n\) such that
\(v_\lambda\tab_\lambda=\tab^\lambda\) is a \(W\!\)-graph
determining element, and its \(W\!\)-graph is isomorphic to the
\(W\!\)-graph of the left cell that contains \(w_{J}\).
\end{thr}

\begin{rem}\label{wgbytab}
Lemma \ref{characterisetstd1} and Theorem \ref{wdetelmA} justify
the procedure used in \cite{howvan:wgraphDetSets} to compute a
\(W\!\)-graph with vertex set indexed by \(\STD(\lambda)\).
\end{rem}

\begin{rem}\label{specht-to-cell}
By Theorem \ref{wdetelmA} the \(\mathcal H\)-module \(\mathscr
S(\mathscr I\!_\lambda,J)\) derived from the \(W\!\)-graph ideal
\(\mathscr I\!_\lambda\) with respect to \(J\) is isomorphic to
the cell module associated with the cell \(\mathcal
C_\lambda=\mathscr I\!_\lambda w_J\). But it is known---see, for
example, \cite[Lemma 3.4]{mcpal:klrepSym}---that this cell module
is isomorphic to the Specht module~\(S^\lambda\) associated with
the partition~\(\lambda\). So we conclude that \(\mathscr
S(\mathscr I\!_\lambda,J)\cong S^\lambda\).
\end{rem}

Let \(\lambda\in\mathcal P(n)\) and let \(J=J_\lambda\) as above.
Recall that \(\mathscr I=D_J\) is a \(W\!\)-graph ideal with
respect to~\(J\), and the associated \(\mathcal H\)-module
\(\mathscr S(\mathscr I,J)\) is the induced module
\(\mathscr{S}_{\!\phi}\), defined in Section~\ref{dcrcells} above.
Since the maximal cell \(\mathcal C_1^J\) of the
\(W\!\)-graph~\(\Gamma(C_J)\) derived from \(\mathscr I\) gives
rise to the \(W\!\)-graph ideal \(\mathscr I\!_\lambda=\{\,v\in
W\mid v\leq_L v_\lambda\,\}\), it follows that the \(\mathcal
H\)-module \(\mathscr S(\mathscr I\!_\lambda,J)\) is isomorphic
to~\(\mathscr{S}_{\!\phi}/\mathcal A(C_J\setminus\mathcal
C_1^J)\). Writing \(f\) for the homomorphism
\(\mathscr{S}_{\!\phi}\to\mathscr S(\mathscr I\!_\lambda,J)\) with
kernel \(\mathcal A(C_J\setminus\mathcal C_1^J)\), it follows from
Lemma~\ref{basis-B} that \((\,f(b^J_w)\mid w\in \mathscr
I\!_\lambda\,)\) is an \(\mathcal A\)-basis of \(\mathscr
S(\mathscr I\!_\lambda,J)\), where the elements \(b^J_w=T_w\otimes
1\) for \(w\in D_J\) make up the basis \(B_J\)
of~\(\mathscr{S}_{\!\phi}\). Now if \(s\in S\) and \(w\in D_J\)
then
\[
T_sb^J_w=\begin{cases}
b^J_{sw}&\text{ if \(sw>w\) and \(sw\in D_J\)}\\
-q^{-1}b^J_w&\text{ if \(sw>w\) and \(sw\notin D_J\)}\\
b^J_{sw}+(q-q^{-1})b^J_w&\text{ if \(sw<w\),}\end{cases}
\]
and if we assume that \(w\in\mathscr I\!_\lambda\) and apply \(f\)
to these formulas we find that
\[
T_sf(b^J_w)=\begin{cases}
f(b^J_{sw})&\text{ if \(s\in\SA(w)\)}\\
-q^{-1}f(b^J_w)&\text{ if \(s\in\WD(w)\)}\\
f(b^J_{sw})+(q-q^{-1})f(b^J_w)&\text{ if
\(s\in\SD(w)\).}\end{cases}
\]
If \(s\in\WA(w)\), so that \(sw\in D_J\setminus\mathscr
I\!_\lambda\), then by Lemma~\ref{genGarnir}
\begin{equation}\label{WAcase}
T_sf(b^J_w)=f(b^J_{sw})=\sum_{\substack{y\in\mathscr
I_{\!\lambda}\\ y<sw}}r^s_{y,w}f(b^J_y)
\end{equation}
for some polynomials \(r^s_{y,w}\in q\mathcal A^+\). Morover,
since \(\mathscr I\!_\lambda\) is generated by the single element
\(v_\lambda\), it follows from
\cite[Proposition~7.9]{howvan:wgraphDetSets} that all the elements
\(y\) appearing in the sum in Equation~(\ref{WAcase}) satisfy
\(y\le w\), and we also know by Corollary~\ref{rywpyw} (see
Equation~(\ref{actsinWAGen})) that \(r^s_{w,w}=q\). So if we now
choose an isomorphism \(\theta\colon \mathscr S(\mathscr
I\!_\lambda,J)\to S^\lambda\) and for each \(t\in\STD(\lambda)\)
define \(b_t=\theta(f(b^J_w))\), where \(w\) is the unique element
of \(\mathscr I\!_\lambda\) such that \(t=w\tab_\lambda\), then we
conclude that \((\,b_t\mid t\in\STD(\lambda)\,)\) is a basis
of~\(S^\lambda\), and for all \(s\in S\) and
\(t\in\STD(\lambda)\),
\[
T_{s}b_{t} =
\begin{cases}
  b_{st}  & \text{if \(s \in \SA(t)\),}\\
  b_{st} + (q - q^{-1})b_{t} & \text{if \(s \in \SD(t)\),}\\
  -q^{-1}b_{t} & \text{if \(s \in \WD(t)\),}\\
  qb_{t} - \sum\limits_{u<t} r^{(s)}_{u,t}b_{s} & \text{if \(s \in \WA(t)\),}
\end{cases}
\]
for some \(r^{(s)}_{u,t}\in q\mathcal A^+\!\). We remark that it
is not hard to deduce that these basis elements are uniquely
determined, to within a scalar multiple, by the conditions that
\(T_sb_{\tab_\lambda}=-q^{-1}b_{\tab_\lambda}\) for all \(s\in
J_\lambda\), and \(b_{w\tab_\lambda}=T_wb_{\tab_\lambda}\) for all
\(w\in\mathscr I\!_\lambda\).

We have thus proved Proposition~6.3
of~\cite{howvan:wgraphDetSets}, the assertion of which was that
the polynomials \(r^{(s)}_{u,t}\) are all divisible by~\(q\).

\section{Acknowledgements}

I want to express my gratitude to A/Professor Robert B. Howlett,
who has helped in numerous fruitful discussions, leading to the
completion of the paper. I wish to thank Professor Andrew Mathas
for showing me several of the crucial references.

\printindex


\begin{thebibliography}{99}

\bibitem{ariki:RSleftcells} S. Ariki, \textit{Robinson-Schensted correspondence and the left cells,}
{Combinatorial Methods in Representation Theory} {Adv. Stud. Pure
Math.} \textbf{28} (2000), 1-20

\bibitem{deo:paraKL} V. V. Deodhar, {On some geometric aspects of Bruhat orderings II.
Parabolic analogue of Kazhdan-Lusztig polynomials,}
\textit{Journal of Algebra} \textbf{111} (1997), 483-506

\bibitem{dipjam:heckA} R. Dipper and G. D. James, {Representations of Hecke
Algebras of General Linear Groups,} \textit{Proc. London Math.
Soc. (3)} \textbf{52}(1986), 20-52

\bibitem{geck:klMurphy} M. Geck, \textit{Kazhdan-Lusztig cells and the Murphy basis,}{Proc. London Math. Soc. (3)}
\textbf{93} (2006), 635-665

\bibitem{howvan:wgraphDetSets} R. B. Howlett and V. M. Nguyen, \textit{$W\!$-\,Graph ideals,} {arXiv:1012.1066} (2010)

\bibitem{kazlus:coxhecke} D. Kazhdan and G. Lusztig, {Representations of Coxeter Groups
and Hecke algebras,} \textit{Invent. Math.} \textbf{53} (1979),
165-184

\bibitem{math:heckeA} A. Mathas, \textit{Iwahori-Hecke Algebras and Schur
Algebras of the Symmetric Group,} {University Lecture Series} (15)
American Math. Soc. Providence (1999)

\bibitem{mcpal:klrepSym} T. P. McDonough and C. A. Pallikaros, {On relations betweeen the classical
and the Kazhdan-Lusztig representations of symmetric groups and
associated Hecke algebras,} \textit{Journal of Pure and Applied
Algebra} \textbf{203} (2005), 133-144

\bibitem{sag:sym} B. E. Sagan, \textit{The Symmetric Group,} Brooks/Cole (1991)

\bibitem{stem:addwgraph} J. R. Stembridge, {Admissible \(W\)\,-Graphs,}
\textit{Representation Theory} \textbf{12} American Mathematical
Society (2008), 346-368.
\end{thebibliography}
\end{document}